\documentclass[11pt,a4paper]{article}
\usepackage{amsmath}
\usepackage{amssymb}
\usepackage{mathrsfs}
\usepackage[runin]{abstract}

\numberwithin{equation}{section}  

\evensidemargin=\oddsidemargin\textwidth=142mm \textwidth=160mm
\newtheorem{theorem}{Theorem}[section]
\newtheorem{lemma}{Lemma}[section]

\newtheorem{corollary}{Corollary}[section]

\newenvironment{proof}[1][Proof]{\begin{trivlist}
\item[\hskip \labelsep {\bfseries #1}]}{\end{trivlist}}

\begin{document}
\title {A Derivative-Hilbert operator Acting on Dirichlet spaces}
\author{ Yun Xu\footnote{ E-mail address: xun\_99\_99@163.com} \quad Shanli Ye\footnote{Corresponding author.~E-mail address:slye@zust.edu.cn}\quad Zhihui Zhou   \\
\small \it School of Science, Zhejiang University of Science and Technology, Hangzhou 310023, China}

 \date{}
\maketitle
\begin{abstract}
Let $\mu$ be a positive Borel measure on the interval $[0,1)$. The Hankel matrix $\mathcal{H}_{\mu}=(\mu_{n,k})_{n,k\geq 0}$ with entries $\mu_{n,k}=\mu_{n+k}$, where $\mu_{n}=\int_{[0,1)}t^nd\mu(t)$, induces  formally the operator as $$\mathcal{DH}_\mu(f)(z)=\sum_{n=0}^\infty\left(\sum_{k=0}^\infty \mu_{n,k}a_k\right)(n+1)z^n ,  z\in \mathbb{D},$$
where $f(z)=\sum_{n=0}^{\infty}a_nz^n$ is an analytic function in $\mathbb{D}$. In this paper, we characterize those positive Borel measures on $[0, 1)$ for which $\mathcal{DH}_\mu$ is bounded (resp. compact) from Dirichlet spaces $\mathcal{D}_\alpha ( 0<\alpha\leq2 )$  into $\mathcal{D}_\beta ( 2\leq\beta<4 )$.\\
{\small\bf Keywords}\quad {Hilbert operator, Dirichlet space, Carleson measure
 \\
  {\small\bf 2020 MR Subject Classification }\quad 47B35, 30H99\\}
\end{abstract}
\maketitle

\section{Introduction}
Let $\mathbb{D}=\{z\in \mathbb{C}:|z|<1\}$ be the open unit disk in the complex plane $\mathbb{C}$ and let $H(\mathbb{D})$ denote the class of all analytic functions in $\mathbb{D}$.
For $0<p<\infty$, the Bergman space $A^p$ consists of all functions $f\in H(\mathbb{D})$ for which
$$\|f\|_{A^p}^{p}=\int_{\mathbb{D}}|f(z)|^{p}dA(z)<\infty,$$
where $dA$ denotes the normalized Lebesgue area measure on $\mathbb{D}$. We refer to \cite{Duren} for more informations about Bergman spaces.

For $\alpha \in \mathbb{R}$, the Dirichlet space $\mathcal{D}_\alpha$ consists of all functions $f(z)=\sum_{n=0}^\infty a_{n}z^n \in H(\mathbb{D})$ for which
$$
\|f\|_{\mathcal{D}_\alpha}^2= \sum_{n=0}^\infty (n+1)^{1-\alpha}|a_{n}|^2<\infty.
$$
We get the classical Dirichlet space $\mathcal{D} = \mathcal{D}_0$  if $\alpha=0$ (See \cite{El}), we get the Hardy space $H^2 = \mathcal{D}_1$ if $\alpha = 1$ (see \cite{Dur,Zhu}), and we obtain the Bergman space $A^2 = \mathcal{D}_2$ if $\alpha = 2$. We mention \cite{El} for a complete information on Dirichlet spaces.

Suppose that $\mu$ is a positive Borel measure on [0,1). We define $H_\mu$ to be the Hankel matrix $(\mu_{n,k})_{n,k\geq0}$ with entries $\mu_{n,k}=\mu_{n+k}$, where $\mu_{n}=\int_{[0,1)}t^nd\mu(t)$. The matrix $\mathcal{H}_\mu$ can be seen as an operator on $f(z)=\sum_{k=0}^\infty a_k z^k\in {H(\mathbb{D})}$ by its action on the Taylor coefficients: $\{a_n\}_{n\geq0} \rightarrow \{\sum_{k=0}^{\infty}\mu_{n,k}a_k\}_{n\geq0}$.  Furthermore, we can formally induce the Hankel operator $\mathcal{H}_\mu$ as
$$
\mathcal{H}_\mu(f)(z)=\sum_{n=0}^\infty(\sum_{k=0}^\infty \mu_{n,k}a_k)z^n ,  z\in \mathbb{D},
$$
whenever the right hand side makes sense and defines an analytic function in $\mathbb{D}.$ If we take the measure to be the Lebesgue measure, $\mathcal{H}_\mu$ is just the Hilbert operator. So $\mathcal{H}_\mu$ is also called generalized Hilbert operator.

The operator $\mathcal{H}_\mu$ have been extensively studied in  \cite{Bao,Cha,Duren,Zhu}. Galanopoulos and Pel$\acute{\text{a}}$ez \cite{Ga} characterized those measures $\mu$ supported on $[0,1)$ such that the generalized Hilbert operator $\mathcal{H}_\mu$ is well defined and is bounded on $\mathcal{H}^1$. Chatzifountas and Girela \cite{Cha} described those measures $\mu$ for which $\mathcal{H}_\mu$ is a bounded operator from $\mathcal{H}^p$ into $\mathcal{H}^q$, where $0<p,q<\infty$. Diamantopoulos \cite{Diama} gave many results about the operator induced by Hankel matrics on Dirichlet space. In 2018, Girela\cite{Gire} introduced the operators $\mathcal{H}_\mu$ acting on certain conformally invariant spaces.

In \cite{Ye,YeS},  the second and the third authors first used the Hankel matrix  defined the Derivative-Hilbert operator  $\mathcal{DH}_\mu$ as
$$\mathcal{DH}_\mu(f)(z)=\sum_{n=0}^\infty(\sum_{k=0}^\infty \mu_{n,k}a_k)(n+1)z^n.$$
It is closed related to the generalized Hilbert operator, that is, $$\mathcal{DH}_\mu(f)(z)=(z\mathcal{H}_\mu(f)(z))'.$$
So we called $\mathcal{DH}_\mu$ to be the Derivative-Hilbert operator. And the second and the third authors characterized the measure $\mu$ for which $\mathcal{DH}_\mu$ is a bounded (resp.compact) operator from $A^p$ into $A^q$ for some $p,~q$ in \cite{YeS}. In \cite{Ye}, they also characterized the measure $\mu$ for which $\mathcal{DH}_\mu$ is a bounded (resp. compact) operator on the Bloch space.

Let us recall the definition of the Carleson-type measure, which is a useful tool for learning about Banach spaces of analytic functions. We refer to \cite{D,Ste} for some results about Carleson measures.

If $I\subset\partial \mathbb{D}$ in an arc, $|I|$ denotes the length of $I$, the Carleson square $S(I)$ is defined as
$$
S(I)=\{z=re^{it}:e^{it}\in I, 1-\frac{|I|}{2\pi}\leq r \leq 1\}.
$$
Suppose $0 < p < \infty$ and $\mu$ is a positive Borel measure on $\mathbb{D}$, then we called $\mu$ to be a $s-$Carleson measure if there exists a positive constant $C$ such that
$$
\sup_{I}\frac{\mu(S(I))}{|I|^s}<\infty, \ for \ any \  interval \  I\subset\partial \mathbb{D}.
$$
We called that $\mu$ is a vanishing $s$-Carleson measure if and only if  $\mu$ satisfies
$$
\lim_{|I|\rightarrow 0}\frac{\mu(S(I))}{|I|^s}=0.
$$

If $\mu$ is a Borel measure on [0,1), it can been seen as a Borel measure on $\mathbb{D}$ by identifying it as $\hat{\mu}(A)=\mu(A\cap[0,1))$, for every Borel set $A\subset \mathbb{D}$. In this way, for $0<s<\infty$, we called $\mu$ to be a $s$-Carleson measure if there exists a positive constant $C$ such that
$$\mu([t,1))\leq C(1-t)^s,\quad t\in[0,1).$$
Also, $\mu$ is a vanishing $s$-Carleson measure on $[0,1]$ if $\mu$ satisfies
$$
\lim_{t\rightarrow 1^{-}}\frac{\mu([t,1))}{(1-t)^s}=0.
$$
Other Carleson type measures on $[0,1)$ have the similar statements.

In this paper, we mainly characterize the positive Borel measure $\mu$ for which the Derivative-Hilbert operator $\mathcal{DH}_{\mu}$ is bounded (resp.compact) from Dirichlet spaces $\mathcal{D}_\alpha ( 0<\alpha\leq2 )$  into $\mathcal{D}_\beta ( 2\leq\beta<4 )$.

Throughout this paper, $C$ denotes a positive constant which depends only on the displayed parameters but not necessarily the same from one occurrence to the next. In addition, we say that $A\gtrsim B$ if there exist a constant $C$ (independent of $A$ and $B$) such that $A\gtrsim CB$, and $A\lesssim B$ is the same as $A\gtrsim B$. In addition, the symbol $A\thickapprox B$ means that $A\lesssim B\lesssim A$.

\section{ Main results}
\hspace*{1.25em} We shall first give a  sufficient condition such that the operator $\mathcal{DH}_{\mu}$ is well defined on the Dirichlet space $\mathcal{D}_{\alpha}$, for $\alpha\in\mathbb{R}$. And we characterize the measure $\mu$ such that $\mathcal{DH}_{\mu}$ is bounded from Dirichlet spaces $\mathcal{D}_\alpha ( 0<\alpha\leq2 )$  into $\mathcal{D}_\beta ( 2\leq\beta<4 )$.
\begin{theorem}\label{Th2.1}
Suppose that $\alpha\in\mathbb{R}$ and let $\mu$ be   a positive Borel measure on $[0,1)$. If
 the proposition $\mu_n=O({n^{-(\frac{\alpha}{2}+\epsilon)}})$ is hold for some $\epsilon>0$, then $\mathcal{DH}_\mu$ is well defined  on $\mathcal{D}_\alpha.$
\end{theorem}
\begin{proof} Suppose $f(z)=\sum_{n=0}^{\infty}a_nz^n\in\mathcal{D}_{\alpha}$. By Cauchy-Schwarz inequality and  \cite[Proposition 1]{Cha}, we obtain that
$$
\begin{aligned}
\left|\sum_{k=0}^\infty\mu_{n,k}a_k\right|&\leq\sum_{k=0}^\infty\mu_{n,k}|a_k|\lesssim\sum_{k=0}^\infty\frac{|a_k|}{(n+k+1)^{\frac{\alpha}{2}+\epsilon}}\\
&\thickapprox\sum_{k=0}^\infty(k+1)^{\frac{\alpha-1}{2}}\frac{1}{(n+k+1)^{\frac{\alpha}{2}+\epsilon}}(k+1)^{\frac{1-\alpha}{2}}|a_k|\\
&\lesssim\left(\sum_{k=0}^\infty\frac{(k+1)^{\alpha-1}}{(n+k+1)^{\alpha+2\epsilon}}\right)^{\frac{1}{2}}\left(\sum_{k=0}^\infty(k+1)^{1-\alpha}|a_k|^2\right)^{\frac{1}{2}}\\
&\lesssim\left(\sum_{k=0}^\infty\frac{1}{(k+1)^{1+2\epsilon}}\right)^{\frac{1}{2}}\|f\|_{\mathcal{D}_\alpha}<\infty.
\end{aligned}
$$
This shows that the operator $\mathcal{DH}_{\mu}$ is well defined on $\mathcal{D}_\alpha.$
\end{proof}

Next, we import an auxiliary lemma to prove the main theorem in this paper.
\begin{lemma}\label{le21}
(\cite[Theorem 318]{Hardy}) Let $K(x,y)$ be a real function of two variables and has the following properties:

$(i)$  $ K(x,y)$ is non-negative and homogeneous of degree -1;

$(ii)$
$$\int_{0}^{\infty}K(x,1)x^{-\frac{1}{2}}dx = \int_{0}^{\infty}K(1,y)y^{-\frac{1}{2}}dy =C;$$

$(iii)$
$K(x,1)x^{-\frac{1}{2}}$ is a strictly decreasing functions of $x$, and $K(1,y)y^{-\frac{1}{2}}$ of $y$; or, more generally;

$(iii')$
$K(x,1)x^{-\frac{1}{2}}$ decreases from $x=1$ onwards, while the interval $(0,1)$ can be divided into two parts, $(0,\xi)$ and $(\xi,1)$, of which one may be null, in the first of which it decreases and in the second of which it increases; and $K(1,y)y^{-\frac{1}{2}}$ has similar properties; and $K(x,x)=0$.\\
Then for every sequence $\{a_n\}_{n\geq0}$ in $l^2$, we get
$$
\sum_{n=1}^\infty \left(\sum_{k=1}^\infty K(n,k)a_k\right)^2\leq C^2 \sum_{n=1}^\infty a_n^2.
$$
\end{lemma}

In short, if $f(z)=\sum_{n=0}^{\infty} {a_n}z^n\in \mathcal{H}^2$, we have
$$\sum_{n=1}^\infty \left(\sum_{k=1}^\infty K(n,k)a_k\right)^2\leq C^2 \|f\|_{\mathcal{H}^2}^2.$$
\begin{theorem}\label{th22}
Suppose that $0<\alpha\leq2$, $2\leq\beta<4$, and let $\mu$ be a positive Borel measure on $[0,1)$ which satisfies the condition in Theorem \ref{Th2.1}.
 Then the following conditions are equivalent:

$(i)$ $\mu$ is a $(2-\frac{\beta-\alpha}{2})$-Carleson measure.

$(ii)$ $\mu_n=O\left(\frac{1}{n^{2-\frac{\beta-\alpha}{2}}}\right).$

$(iii)$ $\mathcal{DH}_\mu$ is bounded operator from $\mathcal{D}_\alpha$ into $\mathcal{D}_\beta.$
\end{theorem}

Before giving the proof, let us recall some classical conclusions about the Beta function.
Let
 $$B(s,t)=\int_{0}^{1}x^{s-1}(1-x)^{t-1}dx,$$
where $s,~t>0$, then we called it to be the Beta function. And we also know that the Beta function as
$$B(s,t)=\int_{0}^{\infty}\frac{x^{s-1}}{(1+x)^{s+t}}dx,$$
where $Re(s)>0,Re(t)>0.$ It is known that the value of $B(s,t)$ is closed related to the Gamma function, that is,
$$B(s,t)=\frac{\Gamma(s)\Gamma(t)}{\Gamma(s+t)}.$$

Now we continue to complete the proof of the Theorem \ref{th22}.
\begin{proof}
$(i)\Rightarrow(ii)$.
Let $\mu$ be a finite positive Borel measure on $[0,1)$, we have
$$
|\mu_n| \leq \int_{0}^1 |t|^n d\mu(t)=n\int_{0}^1 t^{n-1}\mu[t,1)dt
$$
Since $\mu$ is a $(2-\frac{\beta-\alpha}{2})$-Carleson measure, we obtain that
$$
\mu[t,1) \lesssim (1-t)^{2-\frac{\beta-\alpha}{2}}, \quad \text{for \text{any}}\  t\in(0,1).
$$
Hence,$$|\mu_n|\ \lesssim n\int_{0}^{1}t^{n-1} (1-t)^{2-\frac{\beta-\alpha}{2}} dt\thickapprox \frac{1}{n^{2-\frac{\beta-\alpha}{2}}}.$$
As well as, $$\mu_n=O({n^{-(2-\frac{\beta-\alpha}{2})}}).$$

$(ii)\Rightarrow(iii)$.
First, we define two operators. For $f(z)=\sum_{n=0}^\infty a_nz^n\in \mathcal{D}_\alpha$, define $V_\alpha(f)$ by the formula
$$V_\alpha(f)(z)=\sum_{n=0}^{\infty}(n+1)^{\frac{1-\alpha}{2}}a_nz^n. $$
Also,  for $g(z)=\sum_{n=0}^\infty b_nz^n \in \mathcal{H}^2$, define $T_\beta(g)$ by the formula
$$T_\beta(g)(z)=\sum_{n=0}^{\infty}(n+1)^{\frac{\beta-1}{2}}b_nz^n.$$

It is easy to check that $V_\alpha$ is a bounded operator from $\mathcal{D}_\alpha$ into $\mathcal{H}^2$, and $T_\beta$ is a bounded operator from $\mathcal{H}^2$ into $\mathcal{D}_\beta$.

Now suppose that $0<\alpha\leq2$ and $2\leq\beta<4$. For $f(z)=\sum_{n=0}^\infty a_nz^n\in \mathcal{H}^2$, we consider a new operator $S_\mu(f)$ as $$S_\mu(f)(z)=\sum_{n=0}^\infty\left(\sum_{k=0}^\infty(n+1)^{\frac{3-\beta}{2}}(k+1)^{\frac{\alpha-1}{2}}\mu_{n,k}a_k\right)z^n.$$
A direct calculation shows that
$$
\begin{aligned}
\|S_\mu(f)(z)\|_{\mathcal{H}^2}^2&=\sum_{n=0}^\infty \left|\sum_{k=0}^\infty(n+1)^{\frac{3-\beta}{2}}(k+1)^{\frac{\alpha-1}{2}}\mu_{n,k}a_k\right|^2\\
&\leq \sum_{n=0}^\infty\left(\sum_{k=0}^\infty(n+1)^{\frac{3-\beta}{2}}(k+1)^{\frac{\alpha-1}{2}}\mu_{n,k}|a_k|\right)^2\\
&\lesssim \sum_{n=0}^\infty\left(\sum_{k=0}^\infty(n+1)^{\frac{3-\beta}{2}}(k+1)^{\frac{\alpha-1}{2}}\frac{|a_k|}{(n+k+2)^{2-\frac{\beta-\alpha}{2}}}\right)^2\\
&=\sum_{n=1}^\infty\left(\sum_{k=1}^\infty n^{\frac{3-\beta}{2}}k^{\frac{\alpha-1}{2}}\frac{|a_{k-1}|}{(n+k)^{2-\frac{\beta-\alpha}{2}}}\right)^2.
\end{aligned}
$$
Let $$K(x,y)=x^{\frac{3-\beta}{2}}y^{\frac{\alpha-1}{2}}\frac{1}{(x+y)^{2-\frac{\beta-\alpha}{2}}}, \ \ x>0, \ y>0. $$
Then we obtain that
$$
\begin{aligned}
\int_{0}^{\infty}K(x,1)x^{-\frac{1}{2}}dx& = \int_{0}^{\infty}\frac{x^{1-{\frac{\beta}{2}}}}{(x+1)^{2-\frac{\beta-\alpha}{2}}}dx = B(2-\frac{\beta}{2},\frac{\alpha}{2}),\\
\int_{0}^{\infty}K(1,y)y^{-\frac{1}{2}}dy &= \int_{0}^{\infty}\frac{y^{{\frac{\alpha}{2}}-1}}{(y+1)^{2-\frac{\beta-\alpha}{2}}}dy = B(\frac{\alpha}{2},2-\frac{\beta}{2}).
\end{aligned}
$$
And it is clear that the functions $K(x,1)x^{-\frac{1}{2}}, K(1,y)y^{-\frac{1}{2}} $ are strictly decreasing.
Applying Lemma \ref{le21}, it follows that
$$
\sum_{n=1}^\infty\left(\sum_{k=1}^\infty n^{\frac{3-\beta}{2}}k^{\frac{a-1}{2}}\frac{|a_{k-1}|}{(n+k)^{2-\frac{\beta-\alpha}{2}}}\right)^2\lesssim \left(B(2-\frac{\beta}{2},\frac{\alpha}{2})\right)^2 \|f\|_{\mathcal{H}^2}^2.
$$
This implies that the operator $S_\mu$ is bounded on $\mathcal{H}^2$.

For each $f\in \mathcal{D}_\alpha$, it is easy to check that
$$
\begin{aligned}
T_\beta \circ S_\mu \circ V_\alpha(f)(z) &= \sum_{n=0}^\infty\left((n+1)^{\frac{\beta-1}{2}}\sum_{k=0}^\infty(n+1)^{\frac{3-\beta}{2}}(k+1)^{\frac{\alpha-1}{2}}(k+1)^{\frac{1-\alpha}{2}}\mu_{n,k}a_k\right)z^n\\
&=\sum_{n=0}^\infty\left(\sum_{k=0}^\infty\mu_{n,k}a_k\right)(n+1)z^n=\mathcal{DH}_\mu(f)(z),\
\end{aligned}
$$
Hence, $\mathcal{DH}_\mu$ is bounded from $\mathcal{D}_\alpha$ into $\mathcal{D}_\beta$.

$(iii)\Rightarrow(i)$.
For $0 < t < 1,$ let $f_t(z)=(1-t^2)^{1-\frac{\alpha}{2}}\sum_{n=0}^\infty t^nz^n.$ We have that
$$\|f_t\|_{\mathcal{D}_\alpha}^2 = (1-t^2)^{2-\alpha}\sum_{n=0}^\infty(n+1)^{1-\alpha}t^{2
n} \thickapprox 1.$$\\
Therefore,
$$
\begin{aligned}
\| \mathcal{DH}_\mu(f_t)\|_{\mathcal{D}_\beta}^2&\thickapprox \sum_{n=0}^\infty(n+1)^{1-\beta}\left(\sum_{k=0}^\infty(n+1)\mu_{n,k}(1-t^2)^{1-\frac{\alpha}{2}}t^{k}\right)^2\\
&\gtrsim (1-t^2)^{2-\alpha}\sum_{n=0}^\infty(n+1)^{3-\beta}\left(\sum_{k=0}^\infty t^k\int_t^1\chi^{n+k}d\mu(\chi)\right)^2\\
&\gtrsim (1-t^2)^{2-\alpha}\sum_{n=0}^\infty(n+1)^{3-\beta}\left(\sum_{k=0}^n t^{n+2k}\mu[t,1)\right)^2.
\end{aligned}
$$
Since $ \mathcal{DH}_\mu$ is bounded from $\mathcal{D}_\alpha$ into $\mathcal{D}_\beta$,  we obtain that
$$
\begin{aligned}
\| \mathcal{DH}_\mu\|_{\mathcal{D}_\beta}^2\|f_t\|_{\mathcal{D}_\beta}^2&\geq \|\mathcal{DH}_\mu(f_t)\|_{\mathcal{D}_\beta}^2\\
&\gtrsim (1-t^2)^{2-a}\sum_{n=0}^\infty(n+1)^{3-\beta}\left(\sum_{k=0}^n t^{n+2k}\mu[t,1)\right)^2\\
&\gtrsim (1-t^2)^{2-a}\sum_{n=0}^\infty(n+1)^{5-\beta}t^{6n}(\mu[t,1))^2\\
&\thickapprox \frac{(\mu[t,1))^2}{(1-t^2)^{4+\alpha-\beta}}.
\end{aligned}
$$
This implies that
$$\mu[t,1)\lesssim (1-t^2)^{2-\frac{\beta-\alpha}{2}},$$
which is equivalent to saying that $\mu$ is a $(2-\frac{\beta-\alpha}{2})$-Carleson measure.
\end{proof}

In particular, if we take $\alpha=\beta=2$ in Theorem $\ref{th22}$, we can obtain the following corollary which the second and the third authors have proved in \cite{YeS}.
\begin{corollary}
The operator $\mathcal{DH}_\mu$ is bounded on $\mathcal{A}^2$ if and only if the measure $\mu$ is a $2$-Carleson measure.
\end{corollary}

\begin{theorem}
Suppose that $0<\alpha\leq2$, $2\leq\beta<4$, and let $\mu$ be a positive Borel measure on $[0,1)$ which satisfies the condition in Theorem \ref{Th2.1}.
 Then the following conditions are equivalent:

$(i)$ $\mu$ is a vanishing $(2-\frac{\beta-\alpha}{2})$-Carleson measure.

$(ii)$  $\mu_n=o\left(\frac{1}{n^{2-\frac{\beta-\alpha}{2}}}\right).$

$(iii)$ $\mathcal{DH}_\mu$ is compact operator from $\mathcal{D}_\alpha$ into $\mathcal{D}_\beta.$
\end{theorem}
\begin{proof}

$(i)\Rightarrow(ii)$.
It is similar to the previous proof and will not be repeated.

$(ii)\Rightarrow(iii)$.
Take $f(z)=\sum_{n=0}^\infty a_nz^n \in \mathcal{D}_\alpha$. Let
$$S_{\mu ,m}(f)(z)=\sum_{n=0}^m\left(\sum_{k=0}^\infty(n+1)^{\frac{3-\beta}{2}}(k+1)^{\frac{a-1}{2}}\mu_{n,k}a_k\right)z^n,$$
$$\mathcal{DH}_{\mu ,m}(f)(z)=\sum_{n=0}^m\left(\sum_{k=0}^\infty (n+1)\mu_{n,k}a_k\right)z^n.$$
Notice that $S_{\mu ,m}, \mathcal{DH}_{\mu ,m}$ are  finite rank operators, then $S_{\mu ,m}(f)(z)$ is compact on $\mathcal{H}^2.$
Since $\mu_n$ satisfies $\mu_n=o({n^{-(2-\frac{\beta-\alpha}{2})}})$, we obtain that for any $\varepsilon > 0,$ there exists an $N>0$ such that $|\mu_m|<\varepsilon n^{-(2-\frac{\beta-\alpha}{2})}$ when $m>N$. Then we note
$$
(S_\mu-S_{\mu ,m})(f)(z)=\sum_{n=m+1}^\infty\left(\sum_{k=0}^\infty(n+1)^{\frac{3-\beta}{2}}(k+1)^{\frac{a-1}{2}}\mu_{n,k}a_k\right)z^n,
$$
$$
\begin{aligned}
\left(T_\beta \circ S_\mu \circ V_a-T_\beta \circ S_{\mu ,m} \circ V_a\right)(f)(z)
&=\sum_{n=m+1}^\infty\left(\sum_{k=0}^\infty(n+1)\mu_{n,k}a_k\right)z^n\\
&=T_\beta \circ \left(S_\mu-S_{\mu ,m}\right) \circ V_a(f)(z)\\
&=(\mathcal{DH}_\mu-\mathcal{DH}_{\mu ,m})(f)(z).
\end{aligned}
$$
Therefore,
$$
\|(S_\mu-S_{\mu ,m})(f)(z)\|_{\mathcal{H}_2}^2=\sum_{n=m+1}^\infty\left|\sum_{k=0}^\infty (n+1)^{\frac{3-\beta}{2}}(k+1)^{\frac{\alpha-1}{2}}\mu_{n,k}a_k\right|^2.
$$
Then for $m>N,$ we have
$$
\|(S_\mu-S_{\mu ,m})(f)(z)\|_{\mathcal{H}_2}^2\lesssim \varepsilon^2\sum_{n=m+1}^\infty\left(\sum_{k=0}^\infty(n+1)^{\frac{3-\beta}{2}}(k+1)^{\frac{\alpha-1}{2}}\frac{a_k}{(n+k+2)^{2-\frac{\beta-\alpha}{2}}}\right)^2.
$$
By Lemma $\ref{le21}$ and the proof of Theorem $\ref{th22}$, we obtain

$$
\|(S_\mu-S_{\mu ,m})(f)(z)\|_{\mathcal{H}_2}^2\lesssim \varepsilon^2 \|f\|_{\mathcal{H}^2}^2.
$$
Thus,

$$
\|S_\mu-S_{\mu ,m}\|_{\mathcal{H}_2\rightarrow\mathcal{H}_2}\lesssim \varepsilon.
$$
It is clear that
$$
\|\mathcal{DH}_\mu-\mathcal{DH}_{\mu,m}\|_{\mathcal{D}_{\alpha}\rightarrow\mathcal{D}_\beta}
 \lesssim\varepsilon.
$$
Hence, $\mathcal{DH}_\mu$ is compact from $\mathcal{D}_\alpha$ into $\mathcal{D}_\beta$.

$(iii)\Rightarrow(i)$.
For $0 < t < 1,$ let $f_t(z)=(1-t^2)^{1-\frac{a}{2}}\sum_{n=0}^\infty t^nz^n,$ we have
$$\|f_t\|_{\mathcal{D}_\alpha}^2 = (1-t^2)^{2-\alpha}\sum_{n=0}^\infty(n+1)^{1-\alpha}t^{2
n} \thickapprox 1,$$
and $\lim_{t\rightarrow1}f_t(z)=0$ for any $z\in \mathbb{D}$. Since all Hilbert spaces are reflexive, we obtain that $f_t$ is convergent weakly to 0 in $\mathcal{D}_\alpha$ as $t\rightarrow1$. By the assumption that $\mathcal{DH}_\mu$ is compact from $\mathcal{D}_\alpha$ into $\mathcal{D}_\beta,$ we have
$$\lim_{t\rightarrow 1}\|\mathcal{DH}_\mu (f_t)\|_{\mathcal{D}_\beta}=0.$$
Similar to the proof of Theorem \ref{th22}, we obtain that
$$
\mu[t,1)\lesssim (1-t)^{2-\frac{\beta-\alpha}{2}}\|\mathcal{DH}_\mu (f_t)\|_{\mathcal{D}_\beta}.
$$
Therefore,
$$
\lim_{t\rightarrow1}\frac{\mu[t,1))}{(1-t)^{2-\frac{\beta-\alpha}{2}}}=0.
$$
Thus, $\mu$ is a vanishing $(2-\frac{\beta-\alpha}{2})$-Carleson measure.
\end{proof}
\subsection*{Conflicts of Interest}
The authors declare that there are no conflicts of interest regarding the publication of this paper.
\subsection*{Availability of data and material}

The authors declare that all data and material in this paper are available.

\subsection*{Acknowledgments}
The research was supported by the National Natural Science Foundation of China (Grant Nos.11671357).

\end{document}